\documentclass[letterpaper, 10 pt, conference]{ieeeconf}  % Comment this line out
\IEEEoverridecommandlockouts                              % This command is only
\usepackage{graphics} % for pdf, bitmapped graphics files
\usepackage{epsfig} % for postscript graphics files
\usepackage{mathptmx} % assumes new font selection scheme installed
\usepackage{times} % assumes new font selection scheme installed
\usepackage{amsmath} % assumes amsmath package installed
\usepackage{amssymb}  % assumes amsmath package installed
\usepackage{subcaption}
\usepackage{cite}
\usepackage{algorithmic}
\let\labelindent\relax
\usepackage{enumitem}
\newtheorem{proposition}{Proposition}

\newtheorem{algorithm}{Algorithm}

\newtheorem{remark}{Remark}

\title{\LARGE \bf Relaxed Connected Dominating Set Problem \linebreak with Application to Secure Power Network Design}

\author{Kin Cheong Sou and Jie Lu
\thanks{K.C.~Sou is with the Department of Electrical Engineering, National Sun Yat-sen University, Taiwan. J.~Lu is with the School of Science and Technology, ShanghaiTech University, China. E-mail: {\tt sou12@mail.nsysu.edu.tw, lujie@shanghaitech.edu.cn}}}

\begin{document}

\maketitle
\thispagestyle{empty}
\pagestyle{empty}

%%%%%%%%%%%%%%%%%%%%%%%%%%%%%%%%%%%%%%%%%%%%%%%%%%%%%%%%%%%%%%%%%%%%%%%%%%%%%%%%
\begin{abstract}
This paper investigates a combinatorial optimization problem motived from a secure power network design application in [D\'{a}n and Sandberg 2010]. Two equivalent graph optimization formulations are derived. One of the formulations is a relaxed version of the connected dominating set problem, motivating the term relaxed connected dominating set (RCDS) problem. The RCDS problem is shown to be NP-hard, even for planar graphs. A mixed integer linear programming formulation is presented. In addition, for planar graphs a fixed parameter polynomial time solution methodology based on sphere-cut decomposition and dynamic programming is presented. The computation cost of the sphere-cut decomposition based approach grows linearly with problem instance size, provided that the branchwidth of the underlying graph is fixed and small. A case study with IEEE benchmark power networks verifies that small branchwidth are not uncommon in practice. The case study also indicates that the proposed methods show promise in computation efficiency.
\end{abstract}

%%%%%%%%%%%%%%%%%%%%%%%%%%%%%%%%%%%%%%%%%%%%%%%%%%%%%%%%%%%%%%%%%%%%%%%%%%%%%%%%
\section{Introduction}
Our society depends heavily on the proper operation of network systems including intelligent transport systems, electric power distribution and transmission systems etc. These systems are supervised and controlled through Supervisory Control And Data Acquisition (SCADA) systems. For instance, in the electric power transmission grid, SCADA systems collect measurements through remote terminal units (RTUs) and send them to the state estimator to estimate the system states. The estimated states are used for subsequent operations such as contingency analysis (for system health monitoring) and optimal power flow dispatch (for control). Any malfunctioning of these operations can lead to significant social and economical consequences such as the northeast US blackout of 2003.

Because of its importance, the SCADA measurement system has been the subject of extensive studies. Recently, an important measurement system related research topic which has attracted a lot of attention is cyber-physical security. One of the purposes of cyber-physical security studies is to analyze various types of data attacks and their consequences on the system (e.g.,\cite{LRN09,dan2010stealth,sandberg2010security,bobba2010detecting,kosut2010malicious,6504815_TSG,kim2011strategic}). Another important research direction, which is the focus of this paper, is security-guaranteeing system design. A typical design objective is to seek the minimum cost strategic placement of protection resources (e.g., encryption devices, secure phasor measurement units) so that, according to the chosen attack and defense model, no data attack in the system is possible (e.g., \cite{dan2010stealth, bobba2010detecting, kim2011strategic}). The security-guaranteeing system design problem is also closely related to the problem of observability-guaranteeing system design in power systems (e.g., \cite{aminifar2010contingency, chakrabarti2009placement}). Because of the combinatorial feature, it is often considered ``acceptable'' to obtain only suboptimal solutions of protection placement problems. For example, \cite{dan2010stealth, bobba2010detecting, kim2011strategic} consider various types of heuristic algorithms for protection placement, aiming to minimize the protection cost. Reference \cite{6504815_TSG} provides a suboptimal (in economic sense) strategy for some given possible attack scenarios. Instead, this paper reports efficient and {\bf \emph{exact}} solution methodologies, with optimality guarantee, to a nontrivial system protection placement problem first described in \cite{dan2010stealth}. The design problem, to be described in Section~\ref{sec:application_problem}, seeks a minimum cost strategy to encrypt the measurement communications in a power network, in order to prevent stealth data attack of the form in \cite{LRN09}. Reference \cite{dan2010stealth} points out that the design problem is related to a dominating set problem, and proposes a heuristic suboptimal solution algorithm based on the observation. While the analysis in \cite{dan2010stealth} is performed in a linear algebra setting involving matrix rank calculations, this paper investigates the problem from a graph perspective and provides two equivalent graph optimization formulations characterizing the problem. We prove that the design problem is NP-hard (even when restricted to planar graphs). In addition, we derive a mixed integer linear programming formulation of the problem that is easy to implement (with three sets of constraints) and reasonably efficient to solve (e.g., CPLEX solves an instance with 300 nodes in less than one second on a personal computer). To enable the design with very large-scale systems we develop a fixed parameter polynomial time design algorithm, which is a two-step procedure based on sphere-cut decomposition \cite{ROBERTSON1991153} and dynamic programming. This approach provides an exact solution to the design problem when the underlying graph is planar, and provides a (reasonably tight) upper bound in general. The main advantage of the proposed approach is computation efficiency in both theory and practice. The computation cost grows linearly with problem instance size (i.e., number of edges), provided that a graph structure parameter called branchwidth \cite{ROBERTSON1991153} is fixed and small. In practice, it is not uncommon (as indicated by the IEEE power network benchmarks) that the branchwidth of an application graph is small, because intuitively branchwidth is a measure of how closely a graph resembles a tree (branchwidth $\le 2$ for trees). The sphere-cut decomposition (resp., branch decomposition and tree decomposition) approach has been applied with success to provide fixed parameter polynomial time algorithms for difficult combinatorial problems (e.g., \cite{christian2002linear, arnborg1991easy}). In fact, \cite{Sou_ACC2016}, a precursor to this paper, applies the branch decomposition technique to solve the standard dominating set problem related to the design problem in this paper. The main difference between the current paper and \cite{Sou_ACC2016} is that the exact model for the design problem is considered here. In addition, the complexity analysis, the integer programming formulation and the tailored sphere-cut decomposition based optimization algorithm are reported for the first time. The current paper is also similar to two previous work in algorithmic computer science/combinatorial optimization, namely \cite{Marzban2015, Dorn2009}. In particular, reference \cite{Marzban2015}, which is more related to this paper, describes a sphere-cut decomposition/dynamic programming algorithm for the connected dominating set problem. Part of the distinction of this paper is that we consider a relaxed (and more general) version of the connected dominating set problem. As a result, the dynamic programming algorithm needs to be generalized. Moreover, reference \cite{Marzban2015} focuses purely on the theoretical optimization problem, while we formulate the relaxed connected dominating set problem from application.

{\bf Outline:} In Section~\ref{sec:application_problem} the secure system design problem is described. In Section~\ref{sec:formulations_complexity} two equivalent graph optimization  formulations modeling the secure system design problem are presented. The complexity of the problem is discussed, and a mixed integer linear programming formulation is presented. Section~\ref{sec:scd} introduces the two-step fixed parameter polynomial time algorithm for the design problem when the underlying graph is planar. It reviews the first step -- sphere-cut decomposition. Section~\ref{sec:DP} explains the second step -- dynamic programming to solve the design problem. The parameterized complexity is briefly discussed in the end of the section. Section~\ref{sec:case_study} presents a numerical case study on IEEE power network benchmarks. It demonstrates the practical usefulness of the presented solution approaches.

\section{Application motivations} \label{sec:application_problem}
A power network can be modeled as an undirected connected graph where the nodes are buses, and the edges are transmission lines. Following \cite{dan2010stealth}, this paper adopts the DC power flow model \cite{Abur_Exposito_SEbook} as the measurement model for state estimation. In this model, the power system states are the voltage phasors at the buses and the vector of states is denoted by $\theta$. In this paper, the ``full measurement'' assumption is made. That is, as in the setup of \cite{dan2010stealth}, each bus is equipped with a remote terminal unit (RTU) to obtain the following measurements: (net) active power injection at the bus and active power flows on the transmission lines incident to the bus. Let $z$ denote the vector of measurements. Then, the states and measurements are related by $z = H \theta + \Delta z$, where $H$ is the measurement matrix describing how the active power injection and active power flow measurements are linearly related to the voltage phasors (i.e., the states). $\Delta z$ models the imperfection of the measurements. In this paper, $\Delta z$ is assumed to be the vector of data attacks in the measurements.

In power system operations, a ``bad data detection'' (BDD) scheme attempts to detect possible data attack in the measurements (i.e., $\Delta z$). In a typical residual-based BDD scheme, the measurement residual $r$ and the data attack $\Delta z$ are related by $r = (I - H (H^T R^{-1} H)^{-1} H^T R^{-1}) \Delta z := S \Delta z$, where $R$ is a given diagonal positive definite matrix and $S$ is typically referred to as the residual sensitivity matrix. In the BDD scheme, if the residual $r$ is large (in magnitude, for example) the data attack $\Delta z$ is also large. In this case, the operator is notified of possible anomalies in the power system. However, it can be verified that $S H = 0$. This fact was exploited in \cite{LRN09} to introduce a detection-evading data attack of the form $\Delta z = H \tilde{\theta}$, where $\tilde{\theta}$ can be interpreted as a vector of ``fictitious'' voltage phasors, because the residual resulted by $\Delta z = H \tilde{\theta}$ is zero. In view of the interpretation of the measurement matrix $H$ and the full measurement assumption, a data attack $\Delta z$ can evade BDD detection if the following ``attack rules'' are satisfied:

\begin{enumerate}[label=(A\arabic*)]
\item every bus can be associated with some fictitious voltage phasor (collectively forming the vector $\tilde{\theta}$ in above) such that the component of $\Delta z$ affecting active power flow measurement on a transmission line is proportional to the difference of the fictitious voltage phasors at the incident buses.
\item at each bus the component of $\Delta z$ affecting active power injection measurement satisfies Kirchhoff's current law with the components of $\Delta z$ affecting the active power flow measurements on incident transmission lines.
\end{enumerate}
To counter these attacks, \cite{dan2010stealth} considers the scenario in which buses can be protected by installing authentication devices in the corresponding RTUs. The protection rules are as follows: 
\begin{enumerate}[label=(P\arabic*)]
\item if a bus is protected, then none of the measurements related to the bus can be attacked. In other words, the components of data attack $\Delta z$ corresponding to the measurements related to the bus must be zero. The related measurements include the active power injection at the bus and the active power flows on all incident transmission lines.
\item if a bus is not protected, then the active power injection measurement at the bus can be attacked.
\item if a transmission line is not incident to any protected bus, then the active power flow measurement on it can be attacked.
\end{enumerate}
Note that without the ``full measurement'' assumption, protection rules different from (P1)-(P3) need to be considered. The discussion of the ramifications is, however, beyond the scope of this paper. A subset of buses is called a {\bf \emph{perfect protection set}} if when the buses in the set are protected, according to the protection rules (P1) through (P3), there cannot be any detection-avoiding data attack with fictitious voltage phasors satisfying attack rules (A1) and (A2). In \cite{dan2010stealth}, the {\bf \emph{perfect protection problem}} seeks a minimum cardinality perfect protection set. Assuming that the cost associated with the protection is nondecreasing with the number of protected buses, a minimum perfect protection set provides the most economical protection placement strategy. See Fig.~\ref{fig:perfect_protection} for two examples of perfect protection sets of a six-node network.
\begin{figure}[ht]
            \centering
            \parbox{1.2in}{\includegraphics[width=35mm]{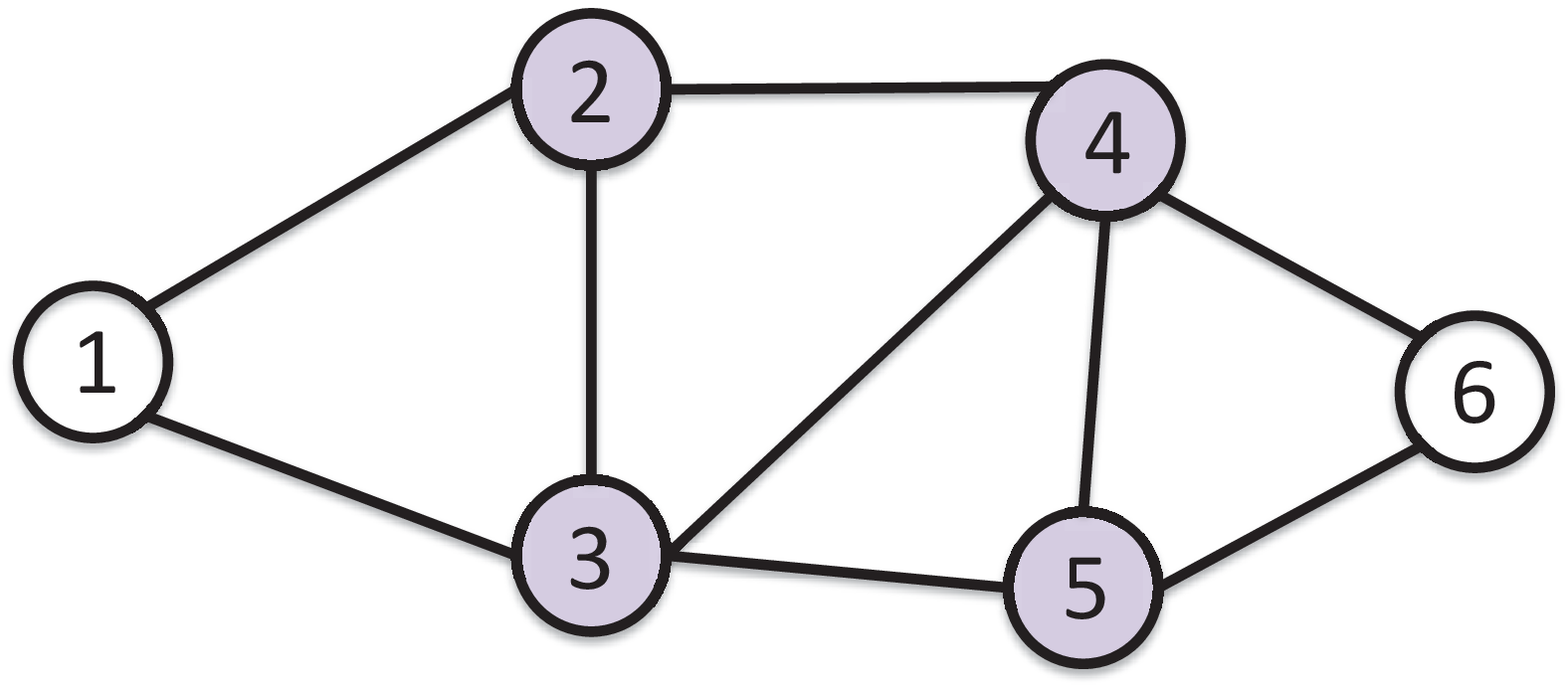}}
            \qquad
            \begin{minipage}{1.2in}
              \includegraphics[width=35mm]{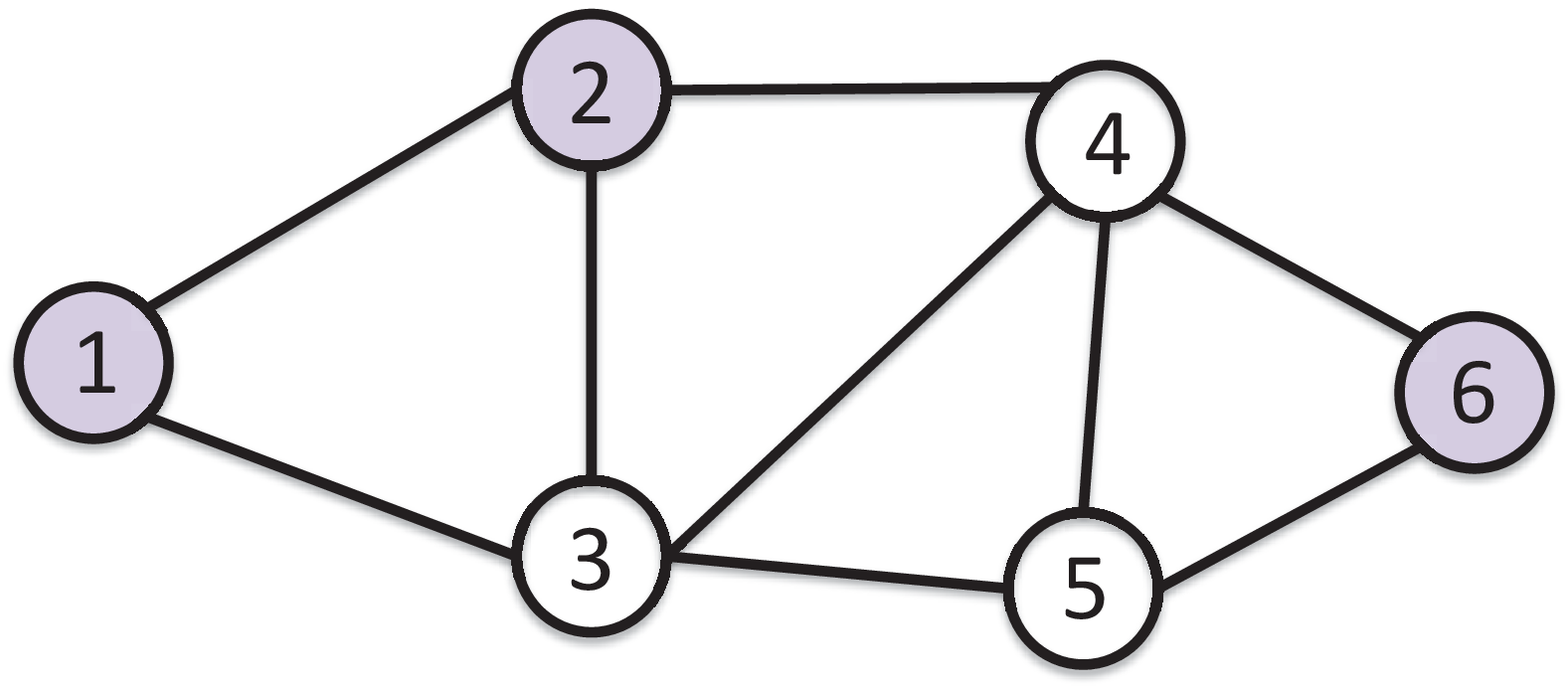}
            \end{minipage}
            \caption{In both figures, the shaded buses are protected and they form perfect protection sets. Left: every transmission line is incident to at least one protected bus. By (P1), the data attack component on every transmission line must be zero. By (A2), the Kirchhoff's current law implies that all injection components of $\Delta z$ must be zero as well. Right: not all components of $\Delta z$ are immediately set to zero by (P1). However, the components of $\Delta z$ are zero on $\{1,2\}$, $\{2,3\}$, $\{2,4\}$,$\{4,6\}$,$\{5,6\}$, forming a spanning tree. By (A1), the fictitious phasors on all buses are the same (in fact zero due to grounding). Hence, nonzero detection-evading data attack is impossible.}
            \label{fig:perfect_protection}
\end{figure}

\section{Problem formulations and complexity} \label{sec:formulations_complexity}

Let $(V,E)$ be a simple undirected graph modeling a power network, with $V$ being the node set and $E \subseteq \{\{u,v\} \mid u \in V, v \in V\}$ being the edge set. In this paper, we use the unordered pair $\{u,v\}$ to denote an (undirected) edge. For any graph $G$, we use the symbols $V(G)$ and $E(G)$ to denote the sets of nodes and edges, respectively. 

A set $D \subseteq V$ is called a {\bf \emph{(graph) perfect protection set}} if the bus set corresponding to $D$ is a perfect protection set for the power network modeled by $(V,E)$. Accordingly, the {\bf \emph{(graph) perfect protection problem}} seeks a minimum cardinality (graph) perfect protection set. For a given graph modeling a power network, it is possible to characterize the perfect protection sets directly in graph, without using the original definition of perfect protection sets involving attack rules (A1) and (A2) and protection rules (P1) through (P3). In the following, two equivalent definitions of perfect protection set are presented. The two definitions lead to two equivalent formulations of the perfect protection problem with different advantages.

\subsection{Integer programming problem formulation}
The first characterization of perfect protection set is as follows:
\begin{proposition} \label{thm:RCDS1}
Let connected graph $(V,E)$ be given. For any $U \subseteq V$, let $I_U(E) := \{\{i,j\} \in E \mid \{i,j\} \cap U \neq \emptyset\}$ denote the subset of $E$ incident to $U$. Then, a set $D \subseteq V$ is a perfect protection set if and only if the subgraph $(V,I_D(E))$ is connected.
\end{proposition}
\begin{proof}
Let $V$ be the set of buses and $E$ be the set of transmission lines for the power network associated with the graph $(V,E)$. If $\{u,v\} \in I_D(E)$, then either bus $u$, $v$ or both is protected. According to protection rule (P1) and transmission line constitutive relation (i.e., attack rule (A1)), the fictitious voltage phasors at $u$ and $v$ are the same. Therefore, if $(V,I_D(E))$ is connected the fictitious voltage phasors at all buses are the same. In this case, non-zero data attack is impossible. Conversely, suppose $(V,I_D(E))$ is not connected. Let $(V_0,E_0)$ be a connected component of $(V,I_D(E))$. Let $\tilde{E} = \{ \{u,v\} \in E \mid u \in V_0, v \in V \setminus V_0\}$. That is, $\tilde{E}$ connects $(V_0,E_0)$ with the rest of the $(V,E)$ if the edges in $\tilde{E}$ were not removed. Note that $\tilde{E} \neq \emptyset$ since $(V,E)$ is assumed to be connected. In addition, $\tilde{E} \cap I_D(E) = \emptyset$, since otherwise $(V_0,E_0)$ would not be a (maximal) connected component in $(V,I_D(E))$. Hence, if $\{u,v\} \in \tilde{E}$ then $u \notin D$ and $v \notin D$, meaning that neither bus $u$ nor bus $v$ is protected. By protection rules (P2) and (P3), the injection measurements at $u$ and $v$, as well as the power flow measurement on transmission line $\{u,v\}$ can be modified by the attack. As a result, a detection avoiding data attack can be constructed as follows: set the fictitious voltage phasors at all buses in $V_0$ to be one and the fictitious voltage phasors at all other buses to be zero. Then, for every $\{u,v\} \in \tilde{E}$ assuming without loss of generality that $u \in V_0$ and $v \notin V_0$, the active power flow attack on transmission line $\{u,v\}$ equals the proportional constant in attack rule (A1), denoted as $H_{uv}$. In addition, there is $H_{uv}$ units of active power injection modification into $u$ and $H_{uv}$ unit of active power extraction modification from $v$. At all edges not in $\tilde{E}$, the active power flow attacks are zero because the fictitious voltage phase angle difference is nonzero if and only if the transmission line is in $\tilde{E}$. Also, attack rule (A2) can be satisfied by setting the active power injection modifications to zero for all buses not incident to $\tilde{E}$. As a result, the desired detection avoiding data attack is constructed when $(V,I_D(E))$ is not connected.
\end{proof}
See Fig.~\ref{fig:perfect_protection} for two examples of perfect protection sets (shaded). In the left, $I_D(E) = E$ and the subgraph $(V,I_D(E))$ is the original graph which is connected. In the right, $I_D(E) = \{ \{1,2\}, \{1,3\}, \{2,3\}, \{2,4\}, \{4,6\}, \{5,6\} \}$. It can be verified that $(V,I_D(E))$ is connected.
\begin{remark}
The connected components of $(V,I_D(E))$ in the statement of Proposition~\ref{thm:RCDS1} correspond to observable islands in the terminology of power network state estimation observability analysis \cite{Abur_Exposito_SEbook}. \hfill $\square$
\end{remark}
%\begin{remark}
%Proposition~\ref{thm:RCDS1} implies that perfect protection sets can be found only in connected graphs. \hfill $\square$
%\end{remark}

A benefit of the perfect protection set characterization in Proposition~\ref{thm:RCDS1} is that it enables the mixed integer linear programming formulation of the perfect protection problem.  For each $i \in V$, we denote the neighborhood $N_i := \{j \mid \{i,j\} \in E\}$. In addition, we designate (arbitrarily) a source node $s \in V$. Then, the perfect protection problem is formulated as:
\begin{equation} \label{opt:RCDS_IP}
\begin{array}{cl}
\underset{x, y}{\text{minimize}} & \sum\limits_{i = 1}^{|V|} \; x_i \vspace{1mm} \\
\text{subject to} & \sum\limits_{j \in N_i} y_{ij} - \sum\limits_{j \in N_i} y_{ji} = -1, \quad \forall i \in V \setminus \{s\} \vspace{2mm} \\
& y_{ij} + y_{ji} \le (|V|-1) (x_i + x_j), \quad \forall \{i,j\} \in E \vspace{2mm} \\
& x_i \in \{0,1\}, \;\; \forall i, \;\; y_{ij} \ge 0, \; y_{ji} \ge 0, \;\; \forall \{i,j\} \in E
\end{array}
\end{equation}
The decision variables are defined such that $i \in V$ is in the perfect protection set $D$ if and only if $x_i = 1$. The decision variables $y_{ij}$ and $y_{ji}$ for each $\{i,j\} \in E$ are auxiliary ``network flow'' variables along the edges in two possible directions, in order to model connectedness of the subgraph $(V,I_D(E))$ in Proposition~\ref{thm:RCDS1}. The first constraint is flow conservation constraint at all nodes except the source $s$. This means that for every $i \in V \setminus \{s\}$ one unit of flow is being shipped from $s$ to $i$. The second constraint, together with the nonnegativity of the flows in the third constraint, specifies that an edge $\{i,j\}$ can be used to ship flows if and only if at least one of its two ends is chosen in the perfect protection set $D$. As a result, the three constraints together model the requirement that $(V,I_D(E))$ is connected.

\subsection{Relaxed connected dominating set problem formulation}

The second (equivalent) characterization of perfect protection set is as follows:
\begin{proposition} \label{thm:RCDS2}
Let graph $(V,E)$ be given. A set $D \subseteq V$ is a perfect protection set if and only if $D$ satisfies both of the following conditions
\begin{enumerate}
\item $D$ is a dominating set of $(V,E)$, meaning that every node in $V$ is either in $D$ or a neighbor of a member of D;
\item For any $i,j \in D$, there exists a sequence $i = i_0, i_1, \ldots, i_p = j$ with $\{i_k, i_{k+1}\} \in E$ for $0 \le k \le (p-1)$. In addition, for all $s,t$ with $0 \le s < t \le p$ such that $i_s, i_t \in D$ but $i_k \notin D$ for all $s < k < t$, it holds that $s = t-2$.
%In addition, if $0 \le s < t \le p$ is such that $i_s, i_t \in D$ but $i_k \notin D$ for all $s < k < t$, then $s = t-2$.
\end{enumerate}
\end{proposition}
\begin{proof}
We denote by $C_1$ the condition that the subgraph $(V,I_D(E))$, as defined in Proposition~\ref{thm:RCDS1}, is connected. On the other hand, we denote by $C_2$ conditions 1) and 2) in Proposition~\ref{thm:RCDS2}. To establish the claim of the statement, it suffices to show that $C_1$ and $C_2$ are equivalent. First, suppose $C_1$ holds for $D$. Suppose $D$ is not a dominating set of $(V,E)$. Then, there exists $v \in V$ such that neither $v$ nor any of its neighbors is in $D$. By the definition of $I_D(E)$ in Proposition~\ref{thm:RCDS1}, the vertex $v$ is isolated in $(V,I_D(E))$. This is a contradiction of $C_1$, and hence $D$ satisfying $C_1$ must be a dominating set of $(V,E)$. Further, since by $C_1$ $(V,I_D(E))$ is connected, for every $i,j \in D$ there is a path in $(V,I_D(E))$ between $i$ and $j$. Since every edge in the path is incident to at least one vertex in $D$ (according to the definition of $I_D(E)$), the sequence of vertices traversed by the path satisfy part 2) of condition $C_2$. Therefore, $C_1$ implies $C_2$. Conversely, suppose $D$ does not satisfy $C_1$ (i.e., $(V,I_D(E))$ is not connected). Assume further that $D$ is a dominating set of $(V,E)$ (otherwise part 1) of $C_2$ fails to hold). Let $(V_0,E_0)$ be a connected component of $(V,I_D(E))$ such that $V_0 \cap D \neq \emptyset$ and let $i$ denote a vertex in $V_0 \cap D$. Since $(V,I_D(E))$ is not connected, $V \setminus V_0 \neq \emptyset$. In addition, $D \cap (V \setminus V_0) \neq \emptyset$ since otherwise $D$ is not dominating. Hence, there exists some $j \in D \cap (V \setminus V_0)$. Let $\tilde{E} = \{\{u,v\} \in E \mid u \in V_0, v \in V \setminus V_0\}$. Then, any path between $i$ and $j$ must traverse an edge in $\tilde{E}$. Let $\{s,t\} \in \tilde{E}$ be a traversed edge, then, according to the analysis in the proof of Proposition~\ref{thm:RCDS1}, $s \notin D$ and $t \notin D$. As a result, any sequence $i = i_0, i_1, \ldots, i_p = j$ with $\{i_k, i_{k+1}\} \in E$ for $0 \le k \le (p-1)$ fails to satisfy part 2) of $C_2$. Hence, if $C_1$ does not hold for $D$, either part 1) or part 2) of $C_2$ must fail to hold. Consequently, $C_2$ implies $C_1$, and hence $C_1$ and $C_2$ are equivalent.
\end{proof}
In Fig.~\ref{fig:perfect_protection} both shaded node sets are dominating. In the left, between every two nodes in the shaded set $D = \{2,3,4,5\}$ there is a path traversing nodes only in $D$, satisfying condition 2) in Proposition~\ref{thm:RCDS2}. In the right, the sequence $1,2,4,6$ ``links together'' all three nodes in the shaded set $D = \{1,2,6\}$. The sequence satisfies condition 2) in Proposition~\ref{thm:RCDS2}, since between 2 and 6 there is only one node (i.e., 4) that is not in $D$.
\begin{remark} \label{rmk:RCDS}
A connected dominating set $D_c$ (e.g., \cite{Du:2012:CDS:2412083}) is a dominating set with an additional property that between any two nodes in $D_c$ there exists a path traversing nodes only in $D_c$ (e.g., left example in Fig.~\ref{fig:perfect_protection}). Condition 2) in Proposition~\ref{thm:RCDS2} is a relaxed notion of connectedness. For $D$ satisfying condition 2), between any two nodes in $D$ there exists a ``relaxed path'' such that between two consecutive members of $D$ along the path there can be one node not in $D$ (e.g., the path $1,2,4,6$ in the right of Fig.~\ref{fig:perfect_protection}). This motivates the term {\bf \emph{``relaxed connected dominating set'' (RCDS)}}, for any $D \in V$ satisfying conditions 1) and 2) in Proposition~\ref{thm:RCDS2}. \hfill $\square$
\end{remark}

\begin{remark}
By Remark~\ref{rmk:RCDS} a connected dominating set is a RCDS, which is a dominating set according to condition 1) in Proposition~\ref{thm:RCDS2}. Hence, for a given graph it holds that domination number $\le$ ``relaxed connected domination number'' $\le$ connected domination number. See Fig.~\ref{fig:ring6} for an example.\hfill $\square$
\end{remark}
\begin{figure}[!h]
\begin{center}
\includegraphics[width=30mm]{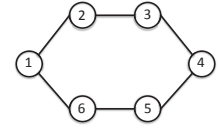}
\caption{In this example graph, a minimum connected dominating set can be $\{1,2,3,4\}$. A minimum RCDS can be $\{1,2,4\}$, while a minimum dominating set can be $\{1,4\}$.}
\label{fig:ring6}
\end{center}
\end{figure}

Proposition~\ref{thm:RCDS2} states that the perfect protection problem can be described as the {\bf \emph{RCDS problem}} seeking a minimum cardinality RCDS. In the sequel, we use the term RCDS problem exclusively to emphasize the graph nature of the problem and its connection to the connected dominating set problem. A proof similar to the one in \cite{NET:NET3230150109} can establish that the RCDS problem is NP-hard even for planar graphs. This proof is obtained by reducing the RCDS problem from the planar vertex cover problem.
\begin{proposition}
The RCDS problem is NP-hard, even if the problem is restricted to instances with planar graphs.
\end{proposition}
\begin{proof}
The proof is shown by a reduction from the planar vertex cover problem whose proof of NP-hardness can be found in, for example, \cite{garey2002computers}. Let $G = (V,E)$ be a planar graph with an arbitrary given embedding, defining an instance of the planar vertex cover problem. We use symbol $G$ both for the planar graph and the plane graph given by the embedding. We assume that $G$ is connected since the minimum vertex cover of a graph is the union of minimum vertex covers of the connected components. We construct an auxiliary bipartite graph $G' = (V',U',E')$ as follows:
\begin{itemize}
\item Vertex set $V' = V \cup F'$, where $F'$ and the set of all faces of $G$ are one-to-one correspondent.
\item Vertex set $U' = A \cup F$. The sets $A$ and $E$ are one-to-one correspondent. The set $F$ and the set of all faces of $G$ are one-to-one correspondent.
\item Edge set $E' = E'_{ve} \cup E'_{vf} \cup E'_f$. For each $e = \{u,v\} \in E$, there are two edges $\{u,e\}$ and $\{e,v\}$ in $E'_{ve}$. For each $f \in F$, there is an edge $\{v,f\} \in E'_{vf}$ with $v \in V$ if and only if $v$ is incident to the face (of $G$) corresponding to $f$. For each face $f$ of $G$, there is an edge $\{f,f'\} \in E'_f$, with $f$, $f'$ corresponding to the same face.
\end{itemize}
By construction, $G'$ is a connected planar graph. The reduction proof consists of two steps. First, we show that for any $V_C \subseteq V$ being a vertex cover of $G$, the subgraph of $G'$ induced by $V_C \cup F$ is connected. It suffices to show that between any two vertices $f, g \in F$ there is a path traversing vertices entirely in $V_C \cup F$, because all vertices in $V$ has at least one neighbor in $F$. We consider the dual graph $G^*$ of $G$. The relevant properties of $G^*$ are as follows (e.g., \cite{diestel2005graph}): 
\begin{itemize}
\item each vertex in $G^*$ corresponds to a face of $G$ (i.e., a member of $F$ in $G'$),
\item each edge connecting two vertices in $G^*$ corresponds to an edge (i.e., a member of $A$ in $G'$) shared by the boundaries of the two faces in $G$,
\item $G^*$ is connected if and only if $G$ is connected.
\end{itemize}
As a result, for $f,g \in F$ (in $G'$), there exists a sequence $(f = )f_1, e_1, f_2, e_2, \ldots, e_{p-1}, f_p( = g)$ for $f_k \in F$ and $e_k \in A$ that corresponds to a path connecting $f$ and $g$ in $G^*$. Since $V_C$ is a vertex cover of $G$, each $e_k$ is covered by (at least) one vertex in $V_C$ denoted by $v_{e_k} \in V_C$. Consequently, the sequence $(f = )f_1, v_{e_1}, f_2, v_{e_2}, \ldots, v_{e_{p-1}}, f_p( = g)$ is a walk in $G'$ traversing vertices entirely in $V_C \cup F$. This establishes the claim of the first step (of the reduction proof).

For the second step, we show that from every minimum RCDS of $G'$ (minimum RCDS exists because $G'$ is connected) it is possible to construct a minimum vertex cover of $G$. To begin, note that any RCDS must include at least one of $f \in F$ and $f' \in F'$ associated with each face of $G$ because the vertices in $F'$ must be dominated. Since it is always more advantageous to contain $f$ than to contain $f'$, from every minimum RCDS it is always possible to construct a (possibly different) minimum RCDS that includes $F$ but not $F'$. We refer to this as claim (a). In turn, the inclusion of $F$ means that the vertices in $V$ are dominated. To dominate the vertices in $A$, it is possible that vertices in $A$ are chosen in a RCDS. However, let $D$ denote a RCDS such that $F \subset D$ and $A \cap D \neq \emptyset$, we can construct another RCDS, denoted $D'$, such that $A \cap D' = \emptyset$ and $D'$ satisfies conditions 1) and 2) in Proposition~\ref{thm:RCDS2} (to be shown shortly). We construct $D'$ by replacing each $\{u,v\} \in A \cap D$ with either $u \in V$ or $v \in V$. By construction $D'$ satisfies condition 1) in Proposition~\ref{thm:RCDS2}. If $\{u,v\} \in D$ and the subsequence $u, \{u,v\}, v$ is on some ``relaxed path'' whose two ends are vertices not in $A$ (this is to establish condition 2) in Proposition~\ref{thm:RCDS2}), then the subsequence $u, \{u,v\}, v$ can be replaced with $u, f_{uv}, v$, where $f_{uv}$ corresponds to a face with which both $u$ and $v$ are incident. The new subsequence is valid for the relaxed (connected) path because $F \subset D'$ (as $F \subset D$). As a result, $D'$ is also a RCDS and $|D'| \le |D|$ (inequality is strict for example when $u \in D$ and it replaces $\{u,v\} \in A \cap D$). We refer to this as claim (b). As a consequence of claims (a) and (b), it is without loss of generality to search for a minimum RCDS for $G'$ with candidates of the form $U \cup F$, where $U \subseteq V$. Though $U \cup F$ need not be a RCDS for some $U \in V$. However, more specialization on $U$ can be inferred. By definition of RCDS, all vertices in $A$ need to be dominated. Because vertices of $F$ are not neighbors of vertices in $A$. The vertices of $A$ must be dominated by the vertices in $U \subseteq V$, meaning that $U$ must be a vertex cover of $G$. Therefore, if a RCDS is of the form $U \cup F$ with $U \subseteq V$ then $U = V_C$ for some vertex cover of $G$. This is referred to as claim (c). In addition, as shown in the first step of the reduction proof, $V_C$ being a vertex cover of $G$ implies that the subgraph (of $G'$) induced by $V_C \cup F$ is connected. This connectedness fact, combined with claims (a), (b) and (c), leads to the conclusion that the RCDS problem is equivalent to a restricted version in which all solution candidates have the form $V_C \cup F$ for $V_C$ being a vertex cover of $G$. Consequently, the minimum objective value of the RCDS problem on $G'$ is $|V_C^\star| + |F|$, where $|V_C^\star|$ is the cardinality of a minimum vertex cover of $G$. This value of the optimal objective value, together with claims (a), (b) and (c), implies that from any optimal solution to the RCDS problem on planar graph $G'$ we can construct a minimum vertex cover for planar graph $G$. This concludes the reduction proof.
\end{proof}

\section{Branchwidth and sphere-cut decomposition} \label{sec:scd}
If the given graph $G = (V,E)$ is {\bf \emph{planar}}, the RCDS problem can be solved in time linear with problem instance size (i.e., $|E|$) when a graph structure parameter called branchwidth (to be defined shortly) is fixed. The proposed approach, which resembles but generalizes the ones in \cite{Marzban2015, Dorn2009}, consists of two steps. Firstly, an optimal sphere-cut decomposition (to be defined) of $G$ is computed. Secondly, a dynamic programming algorithm, based on the computed sphere-cut decomposition, solves the RCDS problem.

Given a graph $(V,E)$, a {\bf \emph{branch decomposition}} \cite{ROBERTSON1991153} is a pair $(T, \tau)$ where $T$ is a unrooted binary tree with $|E|$ leaf nodes, and $\tau$ is a bijection from the set of leaf nodes of $T$ to $E$. Every non-leaf node of $T$ has degree three. For any $e \in E(T)$, the subgraph $(V(T), E(T) \setminus \{e\})$ has two connected components denoted $T_1(e)$ and $T_2(e)$. Let $L_1$ and $L_2$ denote the leaf nodes of $T_1(e)$ and $T_2(e)$, respectively. Then, we denote $E_1(e) := \tau(L_1) \subseteq E$ and $E_2(e) := \tau(L_2) \subseteq E$. We define $G_1(e)$, $G_2(e)$ to be the subgraphs of $G$ induced by $E_1(e)$ and $E_2(e)$ respectively. That is, $G_1(e) = (\mathop{\cup}_{f \in E_1(e)} f, E_1(e))$ and $G_2(e) = (\mathop{\cup}_{f \in E_2(e)} f, E_2(e))$. For $e \in E(T)$, we define the {\bf \emph{middle set}}, denoted $\omega(e)$, to be $V(G_1(e)) \cap V(G_2(e))$. In other words, $\omega(e) := \{v \in V \; \vline \;  v \in e_1, v \in e_2, \; \text{for some $e_1 \in E_1(e)$ and $e_2 \in E_2(e)$} \}$. The width of branch decomposition $(T, \tau)$ is $\max_{e \in E(T)} | \omega(e) |$. The {\bf \emph{branchwidth}} of  $(V,E)$ is the minimum width over all branch decompositions of $(V,E)$. A branch decomposition of $(V,E)$ is optimal if its width is the branchwidth of the graph. For planar graphs, a {\bf \emph{sphere-cut decomposition}} is a branch decomposition with an additional property: for each $e \in E(T)$, it is possible to draw a closed curve separating the subgraphs $G_1(e)$ and $G_2(e)$ in an arbitrary planar embedding of $G$, such that the curve crosses $G$ only at $\omega(e)$. Traversing the curve (clockwise or counterclockwise) leads to a cyclic order denoted by $\pi_e$. For a planar graph $(V,E)$, an optimal sphere-cut decomposition can be computed in $O(\log_2(|V|) (|V|+|E|)^2)$ time (e.g., \cite{seymour1994call}), as the optimal branch decomposition computed using \cite{seymour1994call} is also an optimal sphere-cut decomposition.

\section{\small Dynamic programming for planar RCDS problem} \label{sec:DP}

\subsection{Notations}
%To describe the dynamic programming algorithm, the following definitions are needed:

\begin{description}
\item[$T'$:] Let $(T, \tau)$ be an optimal sphere-cut decomposition of graph $G = (V,E)$. Construct $T'$, a rooted tree from $T$ by inserting two new nodes: (a) node $z$ into any edge $\{u,v\} \in E(T)$, and (b) node $r$ (the root node) forming an edge $\{z,r\}$. Specifically, let $\{u,v\} \in E(T)$, then $T' = \big(V(T) \cup \{z,r\}, (E(T) \setminus \{u,v\}) \cup \{\{u,z\}, \{z,v\}, \{z,r\}\}\big)$. %The rooted tree $T'$ is used to organize the computations in dynamic programming.
\item[$G_e$:] For each $e \in E(T')$, a leaf node $u \in V(T')$ is a descendant of $e$ if the (unique) path from root $r$ to $u$ traverses $e$. Let $V_{T'}(e)$ denote the subset of leaf nodes of $T'$ that are descendants of $e$. Then, $G_e$ is the subgraph of $G$ induced by $\tau(V_{T'}(e))$. $G_e$ is either $G_1(e)$ or $G_2(e)$ in the discussion on branch decomposition in Section~\ref{sec:scd}.
\item[$\omega'$:] Recall the definition of middle set $\omega(e) \subseteq V$ for $e \in E(T)$ in Section~\ref{sec:scd}. We define a function $\omega' : E(T') \mapsto 2^V$ as follows: let $\omega'(\{z,r\}) = \emptyset$, $\omega'(\{u,z\}) = \omega'(\{z,v\}) = \omega(\{u,v\})$ and $\omega'(e) = \omega(e)$ for all $e \in E(T') \setminus \{\{u,z\}, \{z,v\}, \{z,r\}\}$. $\omega'(e)$ is also called the middle set, bordering complement graph of $G_e$.
\item[$e_1$, $e_2$:] For each edge $e \in E(T')$ not incident to a leaf node in $T'$, let $v_e \in V(T')$ denote the node incident to $e$ and having $e$ on the path from the root $r$ to $v_e$. There are two other edges incident to $v_e$. Let $e_1$ and $e_2$ denote these two edges. They are the two children of $e$.
\end{description}

For the example graph $G$ in Fig.~\ref{fig:perfect_protection}, the rooted sphere-cut decomposition tree $T'$ is illustrated in Fig.~\ref{fig:BD_tree}.
\begin{figure}[b]
\begin{center}
\includegraphics[width=65mm]{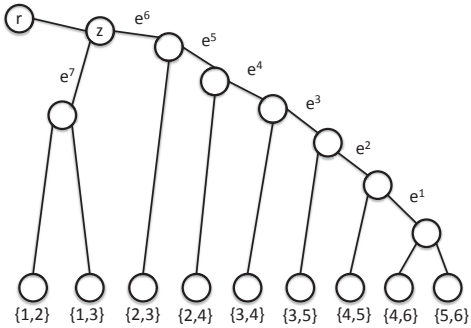}
\caption{The rooted sphere-cut decomposition tree $T'$ of the example graph $G$ in Fig.~\ref{fig:perfect_protection}. Each leaf node of $T'$ corresponds to an edge of $G$.}
\label{fig:BD_tree}
\end{center}
\end{figure}
For $e^3,e^5 \in E(T')$ in Fig.~\ref{fig:BD_tree}, the corresponding subgraphs $G_{e^3} \subseteq G$ and $G_{e^5} \subseteq G$ are shown in Fig.~\ref{fig:GeGf}.
\begin{figure}[h]
            \centering
            \parbox{1.2in}{\includegraphics[width=35mm]{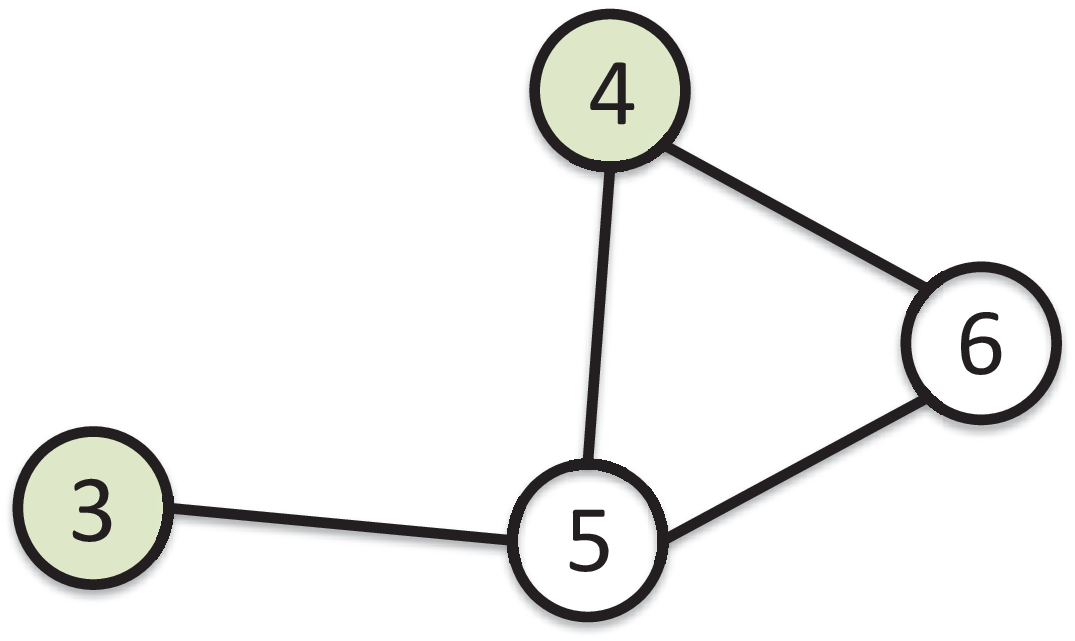}}
            \qquad \qquad
            \begin{minipage}{1.2in}
              \includegraphics[width=35mm]{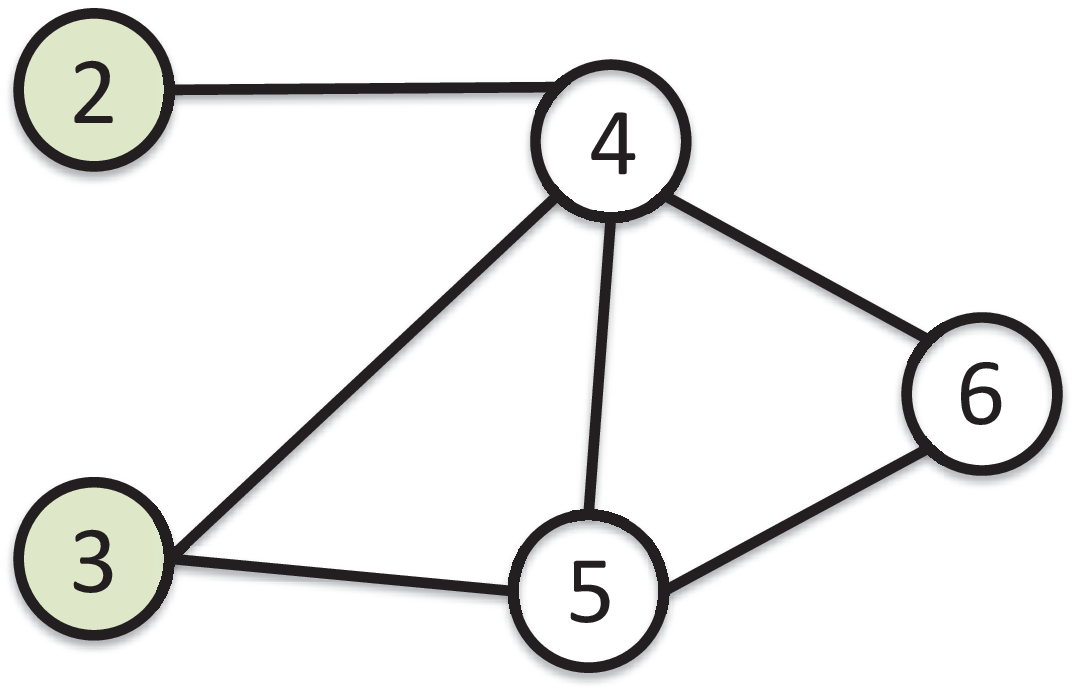}
            \end{minipage}
            \caption{Subgraphs $G_{e^3}$ (left) and $G_{e^5}$ (right) of the graph in Fig.~\ref{fig:perfect_protection}, for $e^3, e^5 \in E(T')$ in Fig.~\ref{fig:BD_tree}. In both subgraphs, the shaded nodes form the middle sets $\omega'(e^3)$ and $\omega'(e^5)$.}
            \label{fig:GeGf}
\end{figure}

\subsection{Dynamic programming procedure} \label{subsec:DP_procedure}

For each $e \in E(T')$, we color the nodes in $\omega'(e)$ in three possible {\bf \emph{basic colors}}: 0, $\hat{0}$, 1. A vector $c_e \in \{0,\hat{0},1\}^{|\omega'(e)|}$ is called a {\bf \emph{basic coloring}} (of $\omega'(e)$). In addition, basic color 0 can be associated with four {\bf \emph{detailed colors}}: $0_[$, $0_]$, $0_*$, $0_s$. Similarly, basic color 1 can be associated with the detailed colors: $1_[$, $1_]$, $1_*$, $1_s$. There is no detailed color associated with basic color $\hat{0}$. Denote $\overline{\mathcal{C}} := \{0_[, 0_], 0_*, 0_s, \hat{0}, 1_[, 1_], 1_*, 1_s\}$. A vector $\bar{c}_e \in \overline{\mathcal{C}}^{|\omega'(e)|}$ is called a {\bf \emph{detailed coloring}} (of $\omega'(e)$). The detailed colorings $\bar{c}_e$ for $e \in E(T')$, supplemented by the basic colorings $c_e$, define the dynamic programming states which parameterize the {\bf \emph{partial problems}} denoted by $P_e(\bar{c}_e)$. A partial problem $P_e(\bar{c}_e)$ seeks a minimum cardinality set $D \subseteq V(G_e)$ satisfying the ``domination constraint'' (details in Appendix~\ref{subsec:DP_consistent}) and the ``relaxed connectedness constraint'' (details in Appendix~\ref{subsec:DP_compatible}). It can be verified (details omitted) that the solution to $P_e(\cdot)$ can be assembled from those of $P_{e_1}(\cdot)$ and $P_{e_2}(\cdot)$ with $e_1$ and $e_2$ being the two children of $e$. In addition, $P_{\{z,r\}}(\cdot)$ is the original RCDS problem. We denote $A_e(\bar{c}_e)$ as the optimal objective value of partial problem $P_e(\bar{c}_e)$, with $A_e(\bar{c}_e) = \infty$ if and only if $P_e(\bar{c}_e)$ is infeasible. $A_e(\cdot)$ is the value function in dynamic programming. For $e \in E(T')$ incident to a leaf node in $T'$, the partial problem $P_e(\cdot)$ involves only one edge and two nodes -- it can be solved by enumeration. Accordingly, for these ``leaf edges'' of $T'$ the value function $A_e(\cdot)$ can be defined with the following rule depending on whether $|\omega'(e)| = 1$ or $|\omega'(e)| = 2$ (no other case is possible). For $|\omega'(e)| = 1$, except for $A_e(0_s) = 1$ and $A_e(1_s) = 1$, all other values of $A_e(\cdot)$ are set to $\infty$. For $|\omega'(e)| = 2$, except for $A_e(0_[,1_]) = 1$, $A_e(\hat{0},\hat{0}) = 0$, $A_e(1_[,0_]) = 1$ and $A_e(1_[,1_]) = 2$, all other values of $A_e(\cdot)$ are set to $\infty$. For each edge in $E(T')$ not incident to the leaf nodes of $T'$, $A_e(\cdot)$ is calculated through the following dynamic programming recursion:
\begin{algorithm}[Dynamic programming recursion]
\label{alg:DP_recursion}
\begin{algorithmic}
\STATE 
\FORALL{basic colorings $c_e \in {\{0,\hat{0},1\}}^{|\omega'(e)|}$}
\FORALL{detailed colorings $\bar{c}_e \in \overline{\mathcal{C}}^{|\omega'(e)|}$ associated with $c_e$}
\STATE{initialize $A_e(\bar{c}_e) = \infty$}
\ENDFOR
\FORALL{basic colorings $(c_{e_1}, c_{e_2})$ consistent with $c_e$}
\FORALL{detailed colorings $(\bar{c}_{e_1}, \bar{c}_{e_2})$ associated with $(c_{e_1}, c_{e_2})$}
\IF{$\bar{c}_{e_1}$ and $\bar{c}_{e_2}$ are compatible, leading to some $\bar{c}_e \in \overline{\mathcal{C}}^{|\omega'(e)|}$}
\STATE If profitable, update $A_e(\bar{c}_e)$ and set $\left(\bar{c}_{e_1}, \bar{c}_{e_2}\right)$ as the pair $\left(\bar{c}_{e_1}^\star(\bar{c}_e),\bar{c}_{e_2}^\star(\bar{c}_e)\right)$ minimizing $A_e(\bar{c}_e)$
\ENDIF
\ENDFOR
\ENDFOR
\ENDFOR
\end{algorithmic}
\end{algorithm}
Certain details of Algorithm~\ref{alg:DP_recursion} require more explanations. Appendix~\ref{subsec:DP_consistent} discusses when basic colorings $(c_{e_1}, c_{e_2})$ are consistent with $c_e$. In addition, Appendix~\ref{subsec:DP_compatible} provides the details about when $\bar{c}_{e_1}$ and $\bar{c}_{e_2}$ are compatible, which $\bar{c}_e$ is resulted and how $A_e(\bar{c}_e)$ is updated. For edge $\{z,r\}$ (i.e., the last iteration), $\omega'(\{z,r\}) = \emptyset$ and hence Algorithm~\ref{alg:DP_recursion} should be interpreted accordingly. In this case, the value function $A_{\{z,r\}}$ is in fact a constant, obtained by the minimization in (\ref{eqn:DP_recursion_zr}) in Appendix~\ref{subsec:DP_compatible_zr}. The details of the last iteration are explained in Appendices~\ref{subsec:DP_consistent_zr} and \ref{subsec:DP_compatible_zr}. 

Let $D^\star$ denote the minimum RCDS to be returned. Then the members of $D^\star$ can be decided by a tree traversal starting from the root: let $e_{r_1}$ and $e_{r_2}$ denote the two children edges of $\{z,r\}$. Let $\bar{c}^*_{e_{r_1}}$, $\bar{c}^*_{e_{r_2}}$ be a minimizing pair leading to the value of $A_{\{z,r\}}$ as returned by Algorithm~\ref{alg:DP_recursion}. Subsequently, for each $e \in E(T')$ such that $\bar{c}^*_e$ is known, we handle two cases:
\begin{itemize}
\item If $e$ is incident to a leaf node of $T'$ then for each $u \in \omega'(e)$, $\bar{c}^*_e(u) \in \{1_[, 1_], 1_*, 1_s\}$ implies $u \in D^\star$ and $u \notin D^\star$ otherwise. If $|\omega'(e)| = 1$ and suppose $u \in \omega'(e)$ and $v \notin \omega'(e)$, then $v \in D^\star$ if and only if $u \notin D^\star$.
\item If $e$ is not incident to a leaf node, then the optimal detailed colorings $\bar{c}^*_{e_1}$ and $\bar{c}^*_{e_2}$ for the two children edges $e_1$ and $e_2$ are set as $\bar{c}_{e_1}^\star(\bar{c}_e^*)$ and $\bar{c}_{e_2}^\star(\bar{c}_e^*)$.
\end{itemize}
The process continues until all edges in $T'$ have been visited.

For the example graph in Fig.~\ref{fig:perfect_protection} and the corresponding rooted sphere-cut decomposition tree $T'$ in Fig.~\ref{fig:BD_tree}, $A_e(\cdot)$ are first determined, in any order, for edges $e$ incident to leaf nodes (i.e., in the bottom). Then, $A_e(\cdot)$ are determined for edges $e = e^1$, \ldots, $e^7$, $\{z,r\}$ one after another. Once $A_e(\cdot)$ is determined for all $e \in E(T')$, the optimal detailed colorings $\bar{c}_e^*$ are determined for edges $e = e^7$, \ldots, $e^1$ one after another. Then the optimal detailed colorings corresponding to the leaf-incident edges are determined and an optimal RCDS $D^\star$ is found.

\subsection{Fixed parameter tractability} \label{subsec:fpt_complexity}
%A (not necessarily optimal) computation cost accounting for the dynamic programming algorithm is as follows: since the computation for retrieving the RCDS in Section~\ref{subsec:DP_RCDS} is a tree-traversal, its computation cost is $O(|E(T')|) = O(|E|)$ (e.g., \cite{fomin2006dominating}). As a result, 
The computation effort of the entire dynamic programming procedure for planar RCDS problem is dominated by the dynamic programming recursion from leaves of $T'$ to root (i.e., Algorithm~\ref{alg:DP_recursion} in Section~\ref{subsec:DP_procedure}). In the innermost for-loop, the computation cost is dominated by the checking of whether $\bar{c}_{e_1}$ and $\bar{c}_{e_2}$ are compatible, which amounts to computing connected components of the bipartite graph detailed in Appendix~\ref{subsec:DP_compatible}. This cost is $O((|\omega'(e_1)| + |(\omega'(e_2)|)^2) = O({\text{BW}}^2)$ time (e.g., \cite{Cormen:2009:IAT:1614191}), where BW is the branchwidth of graph $G$. From each basic coloring $c_{e_x}$ ($x$ being a wildcard for 1 or 2), at most $4^{|\omega'(e_x)|}$ detailed colorings can be derived (e.g., $c_{e_x} = (1,1,\ldots,1)$ leads to $4^{|\omega'(e_x)|}$ detailed colorings). Therefore, the innermost for-loop is executed at most $4^{|\omega'(e_1)| + |\omega'(e_2)|} = O(4^{2 \text{BW}})$ times, for each specific pair of basic colorings $(c_{e_1}, c_{e_2})$. In total, there are no more than $3^{(|\omega'(e)| +  |\omega'(e_1)| + |\omega'(e_2)|)} = O(3^{3\text{BW}})$ pairs of basic colorings $(c_{e_1}, c_{e_2})$ considered in Algorithm~\ref{alg:DP_recursion}. Therefore, Algorithm~\ref{alg:DP_recursion} requires $O({\text{BW}}^2 3^{3 \text{BW}} 4^{2 BW})$ time for each run. Since Algorithm~\ref{alg:DP_recursion} is run for each $e \in E(T')$ between the root and leaves of $T'$ and $|E(T')| = O(|E|)$, the total computation cost for dynamic programming recursion is $O({\text{BW}}^2 3^{3 \text{BW}} 4^{2 BW} |E|)$. This is typical complexity result of a sphere-cut decomposition based algorithm. Computation cost grows very moderately with problem instance size, provided that branchwidth is small. The numerical study in Section~\ref{sec:case_study} indicates that small branchwidth is common in practice.

%of a combined sphere-cut decomposition/dynamic programming solution approach. The computation time linear with respect to the size of the graph (e.g., $|E|$) if the parameter (i.e., branchwidth) does not grow with the size of the graph. On the other hand, the computation cost increases very rapidly with the branchwidth. This manifest the prerequisite for the proposed sphere-cut decomposition/dynamic programming algorithm to be efficient -- the branchwidth of the underlying graph must be small. It turns out that it is not uncommon in practice that branchwidth is small. This is demonstrated in the numerical study in Section~\ref{sec:case_study}. %Finally, it is noted that it is possible refine the computation cost analysis to lead to a less conservative bound similar to the one in \cite{Marzban2015}.

\section{Numerical case study} \label{sec:case_study}
Graphs from the IEEE power system benchmarks are considered. The minimum RCDS of these graphs are computed by solving integer program (\ref{opt:RCDS_IP}). To evaluate the proposed sphere-cut decomposition/dynamic programming based solution approach (SCD approach for short) for planar graphs, we ``planarize'' the benchmark graphs using the algorithm in \cite{Sou_ACC2016} if necessary. Then we compute the corresponding minimum RCDS using the SCD approach. The cardinality of the minimum RCDS on the planarized graph is an upper bound of the cardinality of the minimum RCDS on the original graph, because removing edges makes it more difficult both to satisfy condition 1) and condition 2) in Proposition~\ref{thm:RCDS2}. As a comparison, the minimum dominating sets of the benchmarks are computed to illustrate the difference between the RCDS problem and the dominating set problem. The computation results are listed in Table~\ref{tab:domination_number}. In the fifth column, the symbol $\text{BW}_p$ denotes the branchwidth of the planarized benchmark graph. In the sixth column, the symbol {$|D^\star_{SCD}|$} denotes the cardinality of the minimum RCDS of the planarized benchmark, as returned by the SCD method. This number is greater than or equal to the cardinality of the minimum RCDS of the original graph, which is denoted by $|D^\star|$ in the seventh column. In the last column, the symbol $|DS|$ denotes the regular domination number. It is smaller than or equal to $|D^\star|$, because a RCDS is a dominating set with an additional property (cf.~Proposition~\ref{thm:RCDS2}). Table~\ref{tab:domination_number} verifies the fact that when a graph is planar, the SCD approach computes the exact minimum RCDS. In cases where planarization is necessary, the SCD approach computes only a suboptimal solution.
\begin{table}[h]
\begin{center}
\caption{branchwidth and RCDS cardinality}
\begin{tabular}{|c|c|c|c|c|c|c|c|}
\hline
Name & \scriptsize{$|V|$} & \scriptsize{$|E|$} & planar & $\text{BW}_p$ & \scriptsize{$|D^\star_{SCD}|$} & \scriptsize{$|D^\star|$} & \scriptsize{$|DS|$}  \\[2pt]
\hline
IEEE 9 & 9 & 9 & yes & 2 & 3 & 3 & 3\\
\hline
IEEE 14 & 14 & 20 & yes & 2 & 4 & 4 & 4\\
\hline
IEEE 24 & 24 & 34 & no & 3 & 8 & 8 & 7 \\
\hline
IEEE 30 & 30 & 41 & yes & 3 & 10 & 10 & 10 \\
\hline
IEEE 39 & 39 & 46 & yes & 3 & 15 & 15 & 13 \\
\hline
IEEE 57 & 57 & 78 & no & 4 & 20 & 19 & 17 \\
\hline
IEEE 118 & 118 & 179 & yes & 4 & 34 & 34 & 32 \\
\hline
IEEE 300 & 300 & 409 & no & 4 & 97 & 93 & 87\\
\hline
\end{tabular}
\label{tab:domination_number}
\end{center}
\end{table}

To compute the minimum RCDS, problem (\ref{opt:RCDS_IP}) is solved using IBM CPLEX. To compute the minimum dominating set, an integer program from \cite{Sou_ACC2016} is solved using IBM CPLEX. To compute the branchwidth and sphere-cut decomposition, we utilize the state-of-the-art implementation by Prof.~Gu's group \cite{Bian2016156}. The dynamic programming computations, as described in Section~\ref{subsec:DP_procedure}, are implemented in MATLAB (for ease of implementation). The computations for branchwidth and branch decomposition are performed on a Linux machine with a 3 GHz CPU and 4GB of RAM. All other computations are performed on a Mac machine with a 2.5 GHz CPU and 8GB of RAM. The computation times for solving the RCDS problem using the SCD approach and integer programming (CPLEX) are shown in Table~\ref{tab:time}. In the table, the column labeled $T_{SCD}$ represents the time to compute the sphere-cut decomposition (with the code from Gu's group). The column labeled $T_{DP}$ represents the time to run dynamic programming as described in Section~\ref{subsec:DP_procedure}. The column labeled $T_{IP}$ represents the time to run CPLEX.
\begin{table}[!h]
\begin{center}
\caption{computation time (sec)}
\begin{tabular}{|c|c|c|c|c|c|c|}
\hline
Name & \scriptsize{$|V|$} & \scriptsize{$|E|$} & $\text{BW}_p$ & \scriptsize{$T_{SCD}$} & \scriptsize{$T_{DP}$} & \scriptsize{$T_{IP}$} \\[2pt]
\hline
IEEE 9 & 9 & 9 & 2 & 0.03 & 0.1498 & 0.0047\\
\hline
IEEE 14 & 14 & 20 & 2 & 0.05 & 0.3487 & 0.0064 \\
\hline
IEEE 24 & 24 & 34 & 3 & 0.13 & 1.2433 & 0.0546 \\
\hline
IEEE 30 & 30 & 41 & 3 & 0.06 & 1.3262 & 0.0096 \\
\hline
IEEE 39 & 39 & 46 & 3 & 0.09 & 1.6348 & 0.0127 \\
\hline
IEEE 57 & 57 & 78 & 4 & 0.12 & 21.4133 & 0.3217 \\
\hline
IEEE 118 & 118 & 179 & 4 & 1 & 19.4139 & 0.1287 \\
\hline
IEEE 300 & 300 & 409 & 4 & 3.4 & 45.3790 & 0.6553 \\
\hline
\end{tabular}
\label{tab:time}
\end{center}
\end{table}
From the table, it can be seen that the computation effort of the SCD approach (computing sphere-cut decomposition and dynamic programming) is affected significantly by the branchwidth of the graph, while the effort grows roughly linearly with the problem instance size (i.e., $|E|$). This is consistent with the computation time analyses in \cite{Bian2016156} and in Section~\ref{subsec:fpt_complexity}. However, while the MATLAB implementation of the dynamic programming has acceptable efficiency (within a minute for all test instances), it is far from efficient in contrast with the sphere-cut decomposition implementation and CPLEX. For dynamic programming, a more advanced implementation using a more appropriate language (e.g., C instead of MATLAB) is highly desirable. CPLEX runtime remains reasonable for all the test instances. In a future study, even larger instances should be tested to determine the limitation of integer programming approach.

\section{Conclusion}
This paper is the second installment (the first being \cite{Sou_ACC2016}) of the author's effort to introduce the machinery of graph decomposition (in particular, branch decomposition) to systematically exploit the inherent structure of power networks in practice, in order to develop solution strategies to solve hard combinatorial optimization problems relevant to applications. These two papers demonstrate the flexibility of graph decomposition/dynamic programming as a methodology to derive customized algorithms for application problems. The computation cost of the proposed methodology grows only linearly with problem instance size, provided that the branchwidth is fixed and small. This is drastically different from the scalability property of the integer programming approach. While the computation studies still indicate considerably better efficiency performance for CPLEX, the graph decomposition approach (with the author's amateur implementation) performs reasonably well in practice. The prospect of better scalability of the graph decomposition approach calls for algorithmic developments and implementations in more professional manners in order to truly demonstrate the power of the approach. Further, more theoretical development is desirable. For instance, removing the planarity assumption can dramatically broaden the scope of applications.

\appendix

\subsection{Consistency of basic colorings, general case} \label{subsec:DP_consistent}
For $e \in E(T')$ and $\bar{c}_e \in \overline{\mathcal{C}}^{|\omega'(e)|}$, the associated basic coloring $c_e$ defines a domination constraint for partial problem $P_e(\bar{c}_e)$. In particular, any solution candidate $D \subseteq V(G_e)$ for $P_e(\bar{c}_e)$ must satisfy the ``{\bf \emph{domination constraint}}'' that for each $u \in \omega'(e)$
\begin{itemize}
\item $c_e(u) = 0 \implies$ $u \notin D$ but it is a neighbor of a member of $D$;
\item $c_e(u) = \hat{0} \implies$ $u \notin D$ and it is not neighboring any member of $D$;
\item $c_e(u) = 1 \implies$ $u \in D$.
\end{itemize}
For $e \in E(T')$ and its children edges $e_1$ and $e_2$, the middle sets $\omega'(e)$, $\omega'(e_1)$ and $\omega'(e_2)$ are related. The following definitions are needed to describe the relationship:
\begin{displaymath}
\begin{array}{l}
X_1 := \omega'(e) \setminus \omega'(e_2), \vspace{0.5mm} \\
X_2 := \omega'(e) \setminus \omega'(e_1), \vspace{0.5mm} \\
X_3 := \omega'(e) \cap \omega'(e_1) \cap \omega'(e_2), \vspace{0.5mm} \\
X_4 := (\omega'(e_1) \cup \omega'(e_2)) \setminus \omega'(e).
\end{array}
\end{displaymath}
See Fig.~\ref{fig:omega_ex} for an illustration of the sets defined above.
\begin{figure}[!h]
\begin{center}
\includegraphics[width=60mm]{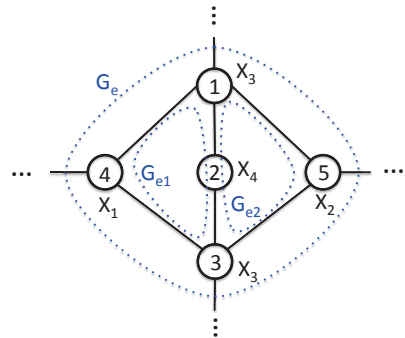}
\caption{In this example, $G_e$ is the graph induced by node set $\{1,2,3,4,5\}$. $G_{e_1}$ is induced by $\{1,2,3,4\}$. $G_{e_2}$ is induced by $\{1,2,3,5\}$. The middle sets are $\omega'(e) = \{1,4,3,5\}$, $\omega'(e_1) = \{1,2,3,4\}$, $\omega'(e_2) = \{1,2,3,5\}$. $X_1 = \{4\}$, $X_2 = \{5\}$, $X_3 = \{1,3\}$, $X_4 = \{2\}$. The node in $X_4$ is in the middle sets of $G_{e_1}$ and $G_{e_2}$, but it becomes an ``internal'' node in $G_e$.}
\label{fig:omega_ex}
\end{center}
\end{figure}
Because of the overlapping of $\omega'(e)$, $\omega'(e_1)$, $\omega'(e_2)$ and the fact that $V(G_e) = V(G_{e_1}) \cup V(G_{e_2})$, the three basic colorings $c_e$, $c_{e_1}$ and $c_{e_2}$ are not always consistent. For a given $c_e$, the pair $(c_{e_1}, c_{e_2})$ are consistent with $c_e$ if there exist RCDS candidates $D_1 \subseteq V(G_{e_1})$ and $D_2 \subseteq V(G_{e_2})$ such that
\begin{itemize}
\item $D_1$ satisfies the domination constraints in $\omega'(e_1)$ characterized by $c_{e_1}$;
\item $D_2$ satisfies the domination constraints in $\omega'(e_2)$ characterized by $c_{e_2}$;
\item For all $u \in X_3 \cup X_4$, either (a): both $u \in D_1$ and $u \in D_2$ or (b): both $u \notin D_1$ and $u \notin D_2$;
\item $D_e := D_1 \cup D_2$ (with $D_e \subseteq V(G_e)$) satisfies the domination constraints in $\omega'(e)$ characterized by $c_e$.
\end{itemize}
Consistency can be characterized directly through the basic colorings $c_e$, $c_{e_1}$ and $c_{e_2}$:
\begin{itemize}
\item For $u \in X_1$, $c_e(u) = c_{e_1}(u)$.
\item For $u \in X_2$, $c_e(u) = c_{e_2}(u)$.
\item For $u \in X_3$, if $c_e(u) \in \{\hat{0},1\}$ then $c_e(u) = c_{e_1}(u) = c_{e_2}(u)$. If $c_e(u) = 0$ then one of the following three holds: (a) $c_{e_1}(u) = 0, c_{e_2}(u) = \hat{0}$ or (b) $c_{e_1}(u) = \hat{0}, c_{e_2}(u) = 0$, or (c) $c_{e_1}(u) = c_{e_2}(u) = 0$.
\item For $u \in X_4$, exactly one of the four cases must hold: (a) $c_{e_1}(u) = c_{e_2}(u) = 1$, (b) $c_{e_1}(u) = 0, c_{e_2}(u) = \hat{0}$, (c) $c_{e_1}(u) = \hat{0}, c_{e_2}(u) = 0$ or (d) $c_{e_1}(u) = c_{e_2}(u) = 0$.
\end{itemize}

\subsection{Consistency of basic colorings, last iteration} \label{subsec:DP_consistent_zr}
For the last iteration of Algorithm~\ref{alg:DP_recursion} concerning the edge $\{z,r\}$, $\omega'(\{z,r\}) = \emptyset$. Hence, $c_e$ is in principle not defined. Accordingly, the first two ``for-loops'' in Algorithm~\ref{alg:DP_recursion} are ignored. Further, the third ``for-loop'' should be understood as $c_{e_1}$ and $c_{e_2}$ being consistent with each other. The consistency in this special case means that for $u \in X_4 = \omega'(e_1) \cup \omega'(e_2)$ exactly one of the four cases must hold: (a) $c_{e_1}(u) = c_{e_2}(u) = 1$, (b) $c_{e_1}(u) = 0, c_{e_2}(u) = \hat{0}$, (c) $c_{e_1}(u) = \hat{0}, c_{e_2}(u) = 0$ or (d) $c_{e_1}(u) = c_{e_2}(u) = 0$. This is a simplified version of the characterization in Appendix~\ref{subsec:DP_consistent} because only $X_4$ is nonempty.

\subsection{Compatibility of detailed colorings and update of $A_e(\cdot)$} \label{subsec:DP_compatible}

Recall from Proposition~\ref{thm:RCDS1} that $D \subseteq V(G_e)$ defines a subgraph $(D, I_D(E(G_e))$. For $e \in E(T')$, a detailed coloring $\bar{c}_e \in \overline{\mathcal{C}}^{|\omega'(e)|}$ defines a ``{\bf \emph{relaxed connectedness constraint}}'' for partial problem $P_e(\bar{c}_e)$. In particular, every solution candidate $D \subseteq V(G_e)$ for $P_e(\bar{c}_e)$ must satisfy (with $x$ being a wildcard for 0 or 1), for all $u \in \omega'(e)$
\begin{itemize}
\item $\bar{c}_e(u) = x_[ \implies c_e(u) = x$. In addition, according to the cyclic order $\pi_e$, $u$ is the first node in the intersection between $\omega'(e)$ and the node set of a connected component of $(D, I_D(E(G_e)))$. In addition, the intersection contains at least two nodes;

\item $\bar{c}_e(u) = x_]$ is similar to the case $\bar{c}_e(u) = x_[$, except that $u$ is the last node in the corresponding intersection;

\item $\bar{c}_e(u) = x_*$ is similar to the case $\bar{c}_e(u) = x_[$, except that $u$ is neither the first nor the last node in the corresponding intersection;

\item $\bar{c}_e(u) = x_s$ is similar to the case $\bar{c}_e(u) = x_[$, except that the corresponding intersection contains only one node, namely $u$.
\end{itemize}
Further, for $\bar{c}_e(u) = \hat{0}$ the domination constraint imposed by basic color $\hat{0}$ must be satisfied by $D$. In this case, the node $u$ colored $\hat{0}$ is not part of any connected components. It can be verified that, because of planarity, a detailed coloring $\bar{c}_e \in \overline{\mathcal{C}}^{|\omega'(e)|}$ specifies completely the information regarding the number of connected components of $(D, I_D(E(G_e)))$ that contain at least one node in $\omega'(e)$, as well as the information regarding the nodes in $\omega'(e)$ that each connected component contains. For $e \in E(T')$ and $e_1$, $e_2$ being the children edges, let $\bar{c}_{e_1}$ and $\bar{c}_{e_2}$ be two detailed colorings such that the associated basic colorings $c_{e_1}$ and $c_{e_2}$ are consistent with some $c_e \in \{0,\hat{0},1\}^{|\omega'(e)|}$. Then, the pair $(\bar{c}_{e_1}, \bar{c}_{e_2})$ are compatible if there exist RCDS candidates $D_1 \subseteq V(G_{e_1})$ and $D_2 \subseteq V(G_{e_2})$ such that
\begin{itemize}
\item $D_1$ satisfies the ``relaxed connectedness'' constraints characterized by $\bar{c}_{e_1}$;
\item $D_2$ satisfies the ``relaxed connectedness'' constraints characterized by $\bar{c}_{e_2}$;
\item For $D_e := D_1 \cup D_2$, every connected component of $(D_e, I_{D_e}(E(G_e)))$ intersecting $X_1 \cup X_2 \cup X_3 \cup X_4$ must intersect $X_1 \cup X_2 \cup X_3 = \omega'(e)$.
\end{itemize}
Note that, with an inductive argument, it can be shown that the third bullet above implies that every connected component of $(D_e, I_{D_e}(E(G_e)))$ intersects $X_1 \cup X_2 \cup X_3 = \omega'(e)$. The following calculations can check whether or not the third bullet above is satisfied: $\bar{c}_{e_1}$ specifies the number of connected components of $(D_1, I_{D_1}(E(G_{e_1})))$ intersecting $\omega'(e_1)$. Let $C_{e_1}^1$, \ldots, $C_{e_1}^p$ denote these connected components. The vector $\bar{c}_{e_1}$ also specifies the nodes in $\omega'(e_1)$ that are contained in each $C_{e_1}^k$ for $k = 1,\ldots,p$. Similarly, $\bar{c}_{e_2}$ specifies the analogous information with the corresponding connected components denoted by $C_{e_2}^1$, \ldots, $C_{e_2}^q$. Construct a bipartite graph with node set $U_1 \cup U_2$. $U_1 = \{u_1,\ldots,u_p\}$ and $U_2 = \{v_1,\ldots,v_q\}$. For $1 \le s \le p$, $1 \le t \le q$, there is an edge $\{u_s, v_t\}$ in the bipartite graph if and only if $C_{e_1}^s$ and $C_{e_2}^t$ share at least one node in $X_3 \cup X_4$. Then, the connected components of the bipartite graph correspond to the connected components of $(D_e, I_{D_e}(E(G_e)))$ that intersect $X_1 \cup X_2 \cup X_3 \cup X_4$. In addition, for each of these connected components (of $(D_e, I_{D_e}(E(G_e)))$) its intersection with $X_1 \cup X_2 \cup X_3 \cup X_4$ can be inferred from the bipartite graph. Thus, the condition in the third bullet can be checked.

If the pair of detailed colorings $(\bar{c}_{e_1}, \bar{c}_{e_2})$ are compatible, there exists a unique detailed coloring $\bar{c}_e \in {\overline{\mathcal{C}}}^{|\omega'(e)|}$ such that all $D_e$ generated by the ``compatibility check'' above (i.e., the three bullets), together with the connected components of its corresponding $(D_e, I_{D_e}(E(G_e)))$, satisfy the ``relaxed connectedness'' constraint defined by $\bar{c}_e$. We call $\bar{c}_{e_1}$ and $\bar{c}_{e_2}$ lead to $\bar{c}_e$. In this case, the optimal solutions of partial problems $P_{e_1}(\bar{c}_{e_1})$ and $P_{e_2}(\bar{c}_{e_2})$ can be combined to form a feasible solution of partial problem $P_e(\bar{c}_e)$. This may lead to an update of the ``currently best'' feasible solution of $P_e(\bar{c}_e)$. Correspondingly, the value function is updated as follows:
\begin{equation} \label{eqn:Jmin_general}
\begin{array}{ccl}
A_e(\bar{c}_e) & \leftarrow & \min \Big\{A_e(\bar{c}_e), \;\; A_{e_1} (\bar{c}_{e_1}) + A_{e_2} (\bar{c}_{e_2}) - \#_1(X_3, \bar{c}_{e_1})  \vspace{0.5mm} \\  & & - \#_1(X_4, \bar{c}_{e_1}) \Big\},
\end{array}
\end{equation}
where $\#_1(X_3, \bar{c}_{e_1})$ denotes the number of nodes in $X_3$ with basic color 1 by coloring $\bar{c}_{e_1}$ (note that $\#_1(X_3, \bar{c}_{e_1}) = \#_1(X_3, \bar{c}_{e_2})$). Similarly, $\#_1(X_4, \bar{c}_{e_1})$ denotes the number of nodes in $X_4$ with basic color 1 by coloring $\bar{c}_{e_1}$ (and also by $\bar{c}_{e_2}$). By definition of $\omega'$, there is no node in exactly one of $\omega'(e)$, $\omega'(e_1)$ and $\omega'(e_2)$. Hence, $X_3$ and $X_4$, which are disjoint, partition the set $\omega'(e_1) \cap \omega'(e_2) = V(G_{e_1}) \cap V(G_{e_2})$ (cf.~Fig.~\ref{fig:omega_ex}). Therefore, in (\ref{eqn:Jmin_general}) the two negative terms prevent double-counting of RCDS members in $\omega'(e_1) \cap \omega'(e_2)$. For the case where $\bar{c}_{e_1}$ and $\bar{c}_{e_2}$ lead to an update of $A_e(\bar{c}_e)$, the minimizing pair is updated as $\bar{c}_{e_1}^\star(\bar{c}_e) = \bar{c}_{e_1}$ and $\bar{c}_{e_2}^\star(\bar{c}_e) = \bar{c}_{e_2}$.

\subsection{Compatibility of detailed colorings, last iteration} \label{subsec:DP_compatible_zr}
For the last iteration of Algorithm~\ref{alg:DP_recursion} concerning the edge $\{z,r\}$, $\omega'(\{z,r\}) = \emptyset$. Hence, no detailed coloring $\bar{c}_{\{z,r\}}$ is defined. Accordingly, the ``if-statement'' in Algorithm~\ref{alg:DP_recursion} should be interpreted as $\bar{c}_{e_1}$ and $\bar{c}_{e_2}$ being compatible. This means that there exist RCDS candidates $D_1 \subseteq V(G_{e_1})$ and $D_2 \subseteq V(G_{e_2})$ such that
\begin{itemize}
\item $D_1$ satisfies the ``relaxed connectedness'' constraints characterized by $\bar{c}_{e_1}$, as described in Appendix~\ref{subsec:DP_compatible};
\item $D_2$ satisfies the ``relaxed connectedness'' constraints characterized by $\bar{c}_{e_2}$, as described in Appendix~\ref{subsec:DP_compatible};
\item For $D := D_1 \cup D_2$, the subgraph $(D, I_{D}(E(G)))$ is connected.
\end{itemize}
This is a simplified and modified version of the characterization in Appendix~\ref{subsec:DP_compatible}. Since $X_3 = \emptyset$, the corresponding value function update is
\begin{equation} \label{eqn:DP_recursion_zr}
A_{\{z,r\}} \leftarrow \min \Big\{A_{\{z,r\}}, \;\; A_{e_1} (\bar{c}_{e_1}) + A_{e_2} (\bar{c}_{e_2}) - \#_1(X_4, \bar{c}_{e_1}) \Big\}.
\end{equation}
Notice that $\#_1(X_4, \bar{c}_{e_1}) = \#_1(X_4, \bar{c}_{e_2})$ because of the consistency of $c_{e_1}$ and $c_{e_2}$.

\bibliographystyle{IEEEtran}
\bibliography{rcds}

% Generated by IEEEtran.bst, version: 1.13 (2008/09/30)
\begin{thebibliography}{10}
\providecommand{\url}[1]{#1}
\csname url@samestyle\endcsname
\providecommand{\newblock}{\relax}
\providecommand{\bibinfo}[2]{#2}
\providecommand{\BIBentrySTDinterwordspacing}{\spaceskip=0pt\relax}
\providecommand{\BIBentryALTinterwordstretchfactor}{4}
\providecommand{\BIBentryALTinterwordspacing}{\spaceskip=\fontdimen2\font plus
\BIBentryALTinterwordstretchfactor\fontdimen3\font minus
  \fontdimen4\font\relax}
\providecommand{\BIBforeignlanguage}[2]{{%
\expandafter\ifx\csname l@#1\endcsname\relax
\typeout{** WARNING: IEEEtran.bst: No hyphenation pattern has been}%
\typeout{** loaded for the language `#1'. Using the pattern for}%
\typeout{** the default language instead.}%
\else
\language=\csname l@#1\endcsname
\fi
#2}}
\providecommand{\BIBdecl}{\relax}
\BIBdecl

\bibitem{LRN09}
Y.~Liu, M.~Reiter, and P.~Ning, ``False data injection attacks against state
  estimation in electric power grids,'' in \emph{16th ACM Conference on
  Computer and Communication Security}, New York, NY, USA, 2009, pp. 21--32.

\bibitem{dan2010stealth}
G.~D{\'a}n and H.~Sandberg, ``Stealth attacks and protection schemes for state
  estimators in power systems,'' in \emph{Smart Grid Communications
  (SmartGridComm), 2010 First IEEE International Conference on}.\hskip 1em plus
  0.5em minus 0.4em\relax IEEE, 2010, pp. 214--219.

\bibitem{sandberg2010security}
H.~Sandberg, A.~Teixeira, and K.~H. Johansson, ``On security indices for state
  estimators in power networks,'' in \emph{First Workshop on Secure Control
  Systems (SCS), Stockholm, 2010}, 2010.

\bibitem{bobba2010detecting}
R.~B. Bobba, K.~M. Rogers, Q.~Wang, H.~Khurana, K.~Nahrstedt, and T.~J.
  Overbye, ``Detecting false data injection attacks on dc state estimation,''
  in \emph{Preprints of the First Workshop on Secure Control Systems, CPSWEEK},
  vol. 2010, 2010.

\bibitem{kosut2010malicious}
O.~Kosut, L.~Jia, R.~J. Thomas, and L.~Tong, ``Malicious data attacks on smart
  grid state estimation: Attack strategies and countermeasures,'' in
  \emph{Smart Grid Communications (SmartGridComm), 2010 First IEEE
  International Conference on}.\hskip 1em plus 0.5em minus 0.4em\relax IEEE,
  2010, pp. 220--225.

\bibitem{6504815_TSG}
A.~Giani, E.~Bitar, M.~Garcia, M.~McQueen, P.~Khargonekar, and K.~Poolla,
  ``Smart grid data integrity attacks,'' \emph{IEEE Transactions on Smart
  Grid}, vol.~4, no.~3, pp. 1244--1253, Sept 2013.

\bibitem{kim2011strategic}
T.~T. Kim and H.~V. Poor, ``Strategic protection against data injection attacks
  on power grids,'' \emph{Smart Grid, IEEE Transactions on}, vol.~2, no.~2, pp.
  326--333, 2011.

\bibitem{aminifar2010contingency}
F.~Aminifar, A.~Khodaei, M.~Fotuhi-Firuzabad, and M.~Shahidehpour,
  ``Contingency-constrained pmu placement in power networks,'' \emph{Power
  Systems, IEEE Transactions on}, vol.~25, no.~1, pp. 516--523, 2010.

\bibitem{chakrabarti2009placement}
S.~Chakrabarti, E.~Kyriakides, and D.~G. Eliades, ``Placement of synchronized
  measurements for power system observability,'' \emph{Power Delivery, IEEE
  Transactions on}, vol.~24, no.~1, pp. 12--19, 2009.

\bibitem{ROBERTSON1991153}
N.~Robertson and P.~Seymour, ``Graph minors. x. obstructions to
  tree-decomposition,'' \emph{Journal of Combinatorial Theory, Series B},
  vol.~52, no.~2, pp. 153 -- 190, 1991.

\bibitem{christian2002linear}
W.~A. Christian~Jr, ``Linear-time algorithms for graphs with bounded
  branchwidth,'' Ph.D. dissertation, Georgia Institute of Technology, 2002.

\bibitem{arnborg1991easy}
S.~Arnborg, J.~Lagergren, and D.~Seese, ``Easy problems for tree-decomposable
  graphs,'' \emph{Journal of Algorithms}, vol.~12, no.~2, pp. 308--340, 1991.

\bibitem{Sou_ACC2016}
K.~C. Sou, ``A branch-decomposition approach to power network design,'' in
  \emph{American Control Conference}, July 2016.

\bibitem{Marzban2015}
M.~Marzban, Q.-P. Gu, and X.~Jia, ``New analysis and computational study for
  the planar connected dominating set problem,'' \emph{Journal of Combinatorial
  Optimization}, pp. 1--28, 2015.

\bibitem{Dorn2009}
F.~Dorn, E.~Penninkx, H.~L. Bodlaender, and F.~V. Fomin, ``Efficient exact
  algorithms on planar graphs: Exploiting sphere cut decompositions,''
  \emph{Algorithmica}, vol.~58, no.~3, pp. 790--810, 2009.

\bibitem{Abur_Exposito_SEbook}
A.~Abur and A.~Exp\'{o}sito, \emph{Power System State Estimation}.\hskip 1em
  plus 0.5em minus 0.4em\relax Marcel Dekker, Inc., 2004.

\bibitem{Du:2012:CDS:2412083}
D.-Z. Du and P.-J. Wan, \emph{Connected Dominating Set: Theory and
  Applications}.\hskip 1em plus 0.5em minus 0.4em\relax Springer Publishing
  Company, Incorporated, 2012.

\bibitem{NET:NET3230150109}
K.~White, M.~Farber, and W.~Pulleyblank, ``Steiner trees, connected domination
  and strongly chordal graphs,'' \emph{Networks}, vol.~15, no.~1, pp. 109--124,
  1985.

\bibitem{garey2002computers}
M.~R. Garey and D.~S. Johnson, \emph{Computers and intractability}.\hskip 1em
  plus 0.5em minus 0.4em\relax wh freeman, 2002, vol.~29.

\bibitem{diestel2005graph}
R.~Diestel, \emph{Graph theory. 2005}.\hskip 1em plus 0.5em minus 0.4em\relax
  Springer, 2005.

\bibitem{seymour1994call}
P.~D. Seymour and R.~Thomas, ``Call routing and the ratcatcher,''
  \emph{Combinatorica}, vol.~14, no.~2, pp. 217--241, 1994.

\bibitem{Cormen:2009:IAT:1614191}
T.~H. Cormen, C.~E. Leiserson, R.~L. Rivest, and C.~Stein, \emph{Introduction
  to Algorithms, Third Edition}, 3rd~ed.\hskip 1em plus 0.5em minus 0.4em\relax
  The MIT Press, 2009.

\bibitem{Bian2016156}
Z.~Bian, Q.-P. Gu, and M.~Zhu, ``Practical algorithms for branch-decompositions
  of planar graphs,'' \emph{Discrete Applied Mathematics}, vol. 199, pp. 156 --
  171, 2016.

\end{thebibliography}

\end{document}